\documentclass{article}
\usepackage{lineno}

\usepackage[USenglish]{babel} 
\usepackage[utf8x]{inputenc}
\usepackage{comment}
\usepackage{authblk}

\usepackage{amsmath}
\usepackage{amsfonts}
\usepackage{amssymb}
\usepackage{xcolor}
\DeclareMathAlphabet{\pazocal}{OMS}{zplm}{m}{n}
\usepackage{cancel}

\usepackage{amsthm}
\usepackage{amsmath,amscd}
\usepackage[all,cmtip]{xy}

\usepackage{indentfirst}
\usepackage{mathrsfs}
\usepackage{stmaryrd}

\usepackage{parskip}
\usepackage{indentfirst}
\usepackage{amsthm}
\usepackage{indentfirst}

\usepackage[normalem]{ulem}
\usepackage[all]{xy}
\usepackage[margin = 1.3in]{geometry}
\usepackage{graphicx,color}
\usepackage{wrapfig}
\usepackage{hyperref} 
\usepackage{geometry}
\usepackage{graphicx,color}
\usepackage{wrapfig}
\usepackage{hyperref}

\usepackage{pb-diagram}
\usepackage{graphics}
\usepackage{tikz-cd}
\usepackage{tikz-qtree}
\usepackage{forest}
\usetikzlibrary{shapes.geometric,backgrounds}
\usepackage{pgf,tikz,pgfplots}\pgfplotsset{compat=1.14}
\usepackage{mathrsfs}
\usetikzlibrary{arrows}

\usepackage{mathabx}
\usepackage{lscape}
\usepackage{geometry}

\usepackage{enumitem}
\usepackage[normalem]{ulem} 

\def\NN{\mathbb{N}}
\def\ZZ{\mathbb{Z}}

\def\ZZ{\mathbb{Z}}

 \newtheorem{defi}{\textbf{Definition}}[section]
 
 \newtheorem{teo}[defi]{\textbf{Theorem}}
 \newtheorem{cor}[defi]{\textbf{Corollary}}
 \newtheorem{prop}[defi]{\textbf{Proposition}}
 \newtheorem{lem}[defi]{\textbf{Lemma}}
 \newtheorem{rem}[defi]{\textbf{Remark}}

  \newtheorem{example}[defi]{\textbf{Example}}

\newcommand{\dbl}{[\hspace{-0.2ex}[}
\newcommand{\dbr}{]\hspace{-0.2ex}]}
\newcommand{\db}[1]{\dbl {#1} \dbr}
\newcommand{\res}[1]{\hspace{-0.6mm}\downarrow_{\hspace{-0.25mm}{#1}}}

\newcommand{\iso}{\cong}

\newcommand{\tn}[1]{\textnormal{#1}}

\begin{document}

\title{Blocks of profinite groups with cyclic defect group}

 \author[1]{Ricardo J. Franquiz Flores}
\author[2]{John W. MacQuarrie}
\affil[1]{Universidade Federal de Minas Gerais, {{rjoel18@gmail.com}}}
\affil[2]{Universidade Federal de Minas Gerais, {{john@mat.ufmg.br}}}

\maketitle
\setcounter{page}{1}

\begin{abstract}
	We demonstrate that the blocks of a profinite group whose defect groups are cyclic have a Brauer tree algebra structure analogous to the case of finite groups.  We show further that the Brauer tree of a block with defect group $\ZZ_{p}$ is of star type.
\end{abstract}


\section{Introduction}

The modular representation theory of a finite group $G$ may loosely be described as the study of the category of $kG$-modules and their relationship with the group $G$, where $k$ for us will be an algebraically closed field whose characteristic $p$ divides the order of $G$.  A simple but effective approach to modular representation theory is to write $kG$ as a direct product of indecomposable algebras, known in the theory as the blocks of $G$, and study the representation theory of each block separately.  The difficulty of a given block $B$ may be measured by a $p$-subgroup of $G$, called the defect group of $B$.  The only general family of defect groups whose corresponding blocks are completely understood, is the class of cyclic $p$-groups.  Blocks with cyclic defect group have an explicit description as so-called ``Brauer tree algebras''.  For a clear and encyclopedic discussion of the block theory of finite groups we recommend \cite{lin1}, \cite{lin2}. 

The study of the modular representation theory of profinite groups was begun in \cite{J}, \cite{J1}, while the study of blocks and defect groups has been initiated recently in \cite{JR}.  In this article we classify, in Theorem \ref{teo:BraTreeAlg}, the blocks of an arbitrary profinite group whose defect groups are cyclic (meaning either a finite cyclic $p$-group or the $p$-adic integers $\mathbb{Z}_p$): they are Brauer tree algebras in strict analogy with the finite case.  Our approach is to use limit arguments and invoke the corresponding theory for finite groups, thus avoiding explicit mention of the most technical arguments required for finite groups.  The results are quite striking: a block of the profinite group $G$ with finite cyclic defect group has finite dimension (Theorem \ref{teo:grpdeffin}) and is thus a block of a finite quotient of $G$ -- the blocks of profinite groups with finite cyclic defect group are thus precisely the blocks of finite groups with cyclic defect group.  Meanwhile, a block with defect group $\mathbb{Z}_p$ has a very simple Brauer tree, of so-called ``star type'' (Theorem \ref{teo:blocktreestar}) -- the blocks with defect group $\mathbb{Z}_p$ are thus well behaved as algebras, even by the standard of blocks of cyclic defect group (Corollary \ref{cor:props of blocks with def Zp}).

\section*{Acknowledgements:}\label{ackref}
We thank the referee for an attentive reading of the text, which has improved the exposition.

\section{Preliminaries}

Let $k$ be a field of characteristic $p$, regarded as a discrete topological ring.  In what follows, when the coefficient ring is unspecified it is assumed to be $k$, so for instance ``algebra" means ``$k$-algebra".  Unless specified otherwise, modules are topological left modules. 

\begin{defi}\label{def:pseudocompactalg}
A \emph{pseudocompact algebra} is an associative, unital, Hausdorff topological $k$-algebra $A$ possessing a basis of open neighbourhoods of $0$ consisting of open ideals $I$ having cofinite dimension in $A$ that intersect in $0$ and such that $A=\varprojlim_{I}A/I$.
\end{defi}

Equivalently, a pseudocompact algebra is an inverse limit of discrete finite dimensional algebras in the category of topological algebras.   For a general introduction to pseudocompact objects, see \cite{Bru}. If $G$ is a profinite group, an inverse system of finite continuous quotients $G/N$ of $G$ induces an inverse system of finite dimensional algebras $k[G/N]$, whose inverse limit $k\db{G}$, the completed group algebra of $G$, is a pseudocompact algebra.

\begin{defi}\label{def:pseudocompactmod}
	If $A$ is a pseudocompact algebra, a \emph{pseudocompact $A$-module} is a topological $A$-module $U$ possessing a basis of open neighbourhoods of $0$ consisting of open submodules $V$ of finite codimension that intersect in $0$ and such that $U=\varprojlim_{V}U/V$. \end{defi}

The category of pseudocompact modules for a pseudocompact algebra is abelian and has exact inverse limits.

\subsection{Induction, restriction, homomorphisms and coinvariants}

The following definitions make use of the \emph{completed tensor product} (cf. \cite[\S 2]{Bru}). Let $G$ be a profinite group and $H$ a closed subgroup of $G$. If $V$ is a pseudocompact $k\db{H}$-module, then the \emph{induced $k\db{G}$-module} is defined by $V\uparrow^{G}=k\db{G}\widehat{\otimes}_{k\db{H}}V$, with action from $G$ on the left factor. If $U$ is a $k\db{G}$-module, then the \emph{restricted $k\db{H}$-module} $U\downarrow_{H}$ is the original $k\db{G}$-module $U$ with coefficients restricted to the subalgebra $k\db{H}$ (cf.\ \cite[\S 2.2]{PPJ}).

Let $N$ be a closed normal subgroup of $G$ and $U$ a pseudocompact $k\db{G}$-module. The module of $N$-\emph{coinvariants} $U_{N}$, is defined as $k\widehat{\otimes}_{k\db{N}}U\cong U/I_{N}U$, where $I_{N}$ denotes the kernel of the continuous map $k\db{N}\to k$ that sends $n$ to $1$ for each $n\in N$. The action of $G$ on $U_{N}$ is given by $g(\lambda\widehat{\otimes}u)=\lambda\widehat{\otimes}gu$. Observe that $U_N$ is naturally a $k\db{G/N}$-module.

The technical properties  we require of coinvariant modules (originally stated for profinite modules in \cite[\S 2]{J}) can be found in \cite[\S 2.3]{JR}.  We mention just one explicitly, wherein the notation $N\trianglelefteq_{O}G$ indicates that $N$ is an open normal subgroup of $G$:

\begin{prop}[{\cite[Proposition 2.7]{JR}}]\label{prop:moduleasinvslimofcoinvar}
	If $U$ is a pseudocompact $k\db{G}$-module, then $\{U_{N}\ :\ N\trianglelefteq_{O}G\}$ together with the canonical quotient maps $\varphi_{MN}:U_{N}\to U_{M}$ whenever $N\leq M$, forms a surjective inverse system with inverse limit $U$.
\end{prop}

Note that if $U$ is finitely generated and $N$ is open, then $U_N$ is a finite dimensional $k[G/N]$-module.

\begin{rem}\label{obs:invsystemforcoinvariants}
	Throughout this text, whenever $N\leq M$ are closed normal subgroups of the profinite group $G$, the notations $\varphi_{MN}:U_N\to U_M,\ \varphi_{N}: U\to U_N$ will be reserved exclusively for the canonical maps between coinvariant modules. In the special case of $k\db{G}$, we have $k\db{G}_{N}= k\db{G/N}$ and the corresponding maps $\varphi_{MN}, \varphi_{N}$ are homomorphisms of algebras.
\end{rem}

In the following lemma the most obvious choice of $W_N$ is simply $\varphi_N(V)$, but allowing more general $W_N$ will be helpful later on.

\begin{lem}\label{lem:quotientpseudocompactmodule}
	Let $U$ be a pseudocompact $k\db{G}$-module and $V$  a closed submodule of $U$. Suppose that for each open normal subgroup $N$ of $G$, $U_{N}$ has a submodule $W_{N}$ such that
	\begin{itemize}
		\item $\varphi_{N}(V)\subseteq W_{N}$ for each $N$,
		\item $\varphi_{MN}(W_{N})\subseteq W_{M}$ whenever $N\leq M$
		\item  $\varprojlim_{N}W_{N}=V$.
	\end{itemize}
	Then $U/V\cong\varprojlim\limits_{N}U_{N}/W_{N}$.
\end{lem}

\begin{proof}
	The conditions imply that the modules $U_N/W_N$ form an inverse system in the obvious way, and that we have a natural map $U\to \varprojlim_N U_N/W_N$.  It is surjective by \cite[IV. \S 3, Lemma 1]{G} and its kernel is $\bigcap_{N}\varphi_{N}^{-1}(W_{N}) = V$.  The induced continuous bijection $U/V\to \varprojlim_N U_N/W_N$ is an isomorphism by \cite[II, \S6, 27.]{L}. 
\end{proof}

Let $A$ be a pseudocompact algebra. If $U$ and $W$ are topological $A$-modules, then $\tn{Hom}_{A}(U,W)$ denotes the topological $k$-vector space of continuous $A$-module homomorphisms from $U$ to $W$ with the compact-open topology. If $U,W$ are pseudocompact, and $W=\varprojlim_{i}W_{i}$, then $\tn{Hom}_{A}(U,W)=\varprojlim_{i}\tn{Hom}_{A}(U,W_{i})$. In particular, when $U$ is finitely generated as an $A$-module, then $\tn{Hom}_{A}(U,W)$ is a pseudocompact vector space. For more details see \cite[§2.2]{PPJ}. When $A=k\db{G}$, $\tn{Hom}_{k\db{G}}(U,W)\cong\varprojlim_{N}\tn{Hom}_{k\db{G}}(U_{N},W_{N})$ (cf. \cite[Remark 2.8]{JR}). 


\subsection{Radicals and socles of pseudocompact $k\db{G}$-modules}

\begin{defi}\label{DefRad}
	Let $A$ be a pseudocompact $k$-algebra and let $U$ be a pseudocompact $A$-module. The \emph{radical of $U$}, denoted by $\tn{Rad}(U)$, is the intersection of the maximal open submodules of $U$.  For each $i\geq 2$ define $\tn{Rad}^i(U)$ recursively to be $\tn{Rad}(\tn{Rad}^{i-1}(U))$, where $\tn{Rad}^1(U) = \tn{Rad}(U)$.  For convenience, we also define $\tn{Rad}^0(U) = U$.
\end{defi}

\begin{lem}\label{lem:charactofradicalofamodule}
	Let $U$ be pseudocompact $k\db{G}$-module with $U=\varprojlim_{N}\{U_{N},\varphi_{MN}\}$. Then
	\begin{enumerate}
		\item For each $i\geq 1$, $\tn{Rad}^{i}(U)=\varprojlim_{N}\tn{Rad}^{i}(U_{N})$ and the maps of the inverse system are surjective.
		\item For each $i\geq 0$, $\tn{Rad}^{i}(U)/\tn{Rad}^{i+1}(U)\cong \varprojlim_{N}\frac{\tn{Rad}^{i}(U_{N})}{\tn{Rad}^{i+1}(U_{N})}$.
	\end{enumerate}
	
\end{lem}

\begin{proof}
The case $i=1$ of the first part follows from routine checks, using that $\varphi_{MN}(\tn{Rad}(U_{N}))= \tn{Rad}(U_{M})$ whenever $N\leq M$ by \cite[Proposition 9.15]{And}, while the general case follows by induction on $i$.  The second part follows from Lemma \ref{lem:quotientpseudocompactmodule}.
\end{proof}

\begin{defi}\label{def:semisimplemodule}
	Let $A$ be a pseudocompact algebra and $U$ a pseudocompact $A$-module. We say that $U$ is \emph{semisimple} if every closed submodule $W$ of $U$ has a closed complement -- that is, there is a closed submodule $W'$ of $U$ such that $U=W\oplus W'$.
\end{defi}

\begin{lem}[{\cite[Lemma 3.9]{Iov}}]\label{lem:equivdefsemisimplicity}
	Let $A$ be a pseudocompact algebra and $U$ a pseudocompact $A$-module. The following are equivalent:
	\begin{enumerate}
		\item $U$ is semisimple.
		\item Every open submodule of $U$ has a complement.
		\item $U$ is a direct product of simple modules.
	\end{enumerate}
\end{lem}

Recall that the \emph{socle} $\tn{Soc}(V)$ of a finite dimensional module $V$ is the largest semisimple submodule of $V$.  If $N\leq M$ are open normal subgroups of $G$ and $X$ is a simple submodule of $U_{N}$, then
$\varphi_{MN}(X)$ is simple or zero, so that $\varphi_{MN}(\tn{Soc}(U_{N}))\subseteq \tn{Soc}(U_{M})$.

\begin{lem}\label{lem:charactofsocleofamodule}
	Let $G$ be a profinite group and $U$ a pseudocompact $k\db{G}$-module. The restriction of the inverse system $\{U_N,\varphi_{MN}\}$ to the socles yields an inverse system, whose limit is the maximal closed semisimple submodule of $U$. 
\end{lem}

\begin{proof}
	Denote by  $L=\varprojlim_{N}\{\tn{Soc}(U_{N}),\varphi_{MN}|_{\tn{Soc}(U_N)}\}$.  We first confirm that $L$ is a semisimple submodule of $U$. Let $W$ be an open submodule of $L$. By Lemma \ref{lem:equivdefsemisimplicity}, it is sufficient to confirm that the canonical projection $\pi:L\to L/W$ splits.  But since $L/W$ is finite dimensional, $\pi$ factors through $\tn{Soc}(U_N)$ for some cofinal subset of $N\trianglelefteq_O G$, as $\pi_N\varphi_N$.  Since $\tn{Soc}(U_N)$ is semisimple and $\pi_N$ is surjective, there is a non-empty linear variety $I_N = \{\iota_N : L/W \to \tn{Soc}(U_N)\,|\,\pi_N\iota_N = \tn{id}_{L/W}\}\subseteq\tn{Hom}_{k\db{G}}(L/W,\tn{Soc}(U_N))$, which is necessarily closed since the latter is finite dimensional.  The $I_N$ form an inverse system in the obvious way, the limit is non-empty by \cite[Lemma 2.3]{JR}, and an element of the limit is a splitting of $\pi$.
	
	
	To prove the lemma, it remains to check that every closed semisimple submodule $V$ of $U$ is contained in $L$.  By Lemma \ref{lem:equivdefsemisimplicity}, $V$ is a direct product of simple modules. It is thus sufficient to confirm that every simple submodule $S$ of $V$ is contained in $L$.  But $\varphi_{N}(S)\subseteq \tn{Soc}(U_{N})$ for each $N$ and so $S\subseteq\varprojlim_{N}\tn{Soc}(U_{N})=L$. Hence, $V\subseteq L$.   
\end{proof}

\begin{defi}\label{def:socleofamodule}
	The \emph{socle} of $U$, denoted by $\tn{Soc}(U)$, is the maximal closed semisimple submodule of $U$. 
\end{defi}

Observe that the inverse system of socles need not be surjective, and it can happen that $\tn{Soc}(U)=0$.

\begin{lem}\label{lem:quotientUoverSocUasinverselimit}
	Let $G$ be a profinite group and $U$ a pseudocompact $k\db{G}$-module. Then $U/\tn{Soc}(U)\cong\varprojlim_{N}U_{N}/\tn{Soc}(U_{N})$.
\end{lem}

\begin{proof}
	The result follows by applying Lemma \ref{lem:quotientpseudocompactmodule} with $V=\tn{Soc}(U)$ and $W_{N}=\tn{Soc}(U_{N})$.
\end{proof}

\subsection{Blocks and defect groups}

Let $k$ be a field of characteristic $p$ and $A$ a pseudocompact $k$-algebra. An element $e\in A$ is \emph{idempotent} if $e^{2}=e$. Two idempotents $e,f$ of $A$ are \emph{orthogonal} if $ef=fe=0$. A non-zero idempotent is \emph{primitive} if it cannot be written as the sum of two non-zero orthogonal idempotents. We denote by $Z(A)$ the center of $A$. An idempotent $e\in A$ is called \emph{centrally primitive} if $e$ is a primitive idempotent of $Z(A)$.

By \cite[IV. \S 3, Corollaries 1,2]{G}, there is a set of pairwise orthogonal centrally primitive idempotents $E=\{e_{i}\mbox{ : }i\in I\}$ in $A$ such that
\begin{eqnarray}\label{eq:blockdecomp}
	A = \prod\limits_{i\in I}Ae_{i} = \prod\limits_{i\in I}B_{i}.
\end{eqnarray}
Each $B_{i}$ is an algebra with unity $e_{i}$. We call $B_{i}$ a \emph{block of $A$} and $e_{i}$  a \emph{block idempotent}. 

We give a brief reminder of the defect group of a block.  For more details and the results mentioned in this paragraph, see \cite[\S 5]{JR}.  The group $G$ acts continuously on $k\db{G}$ by conjugation (notation $g\cdot x = {}^gx$).  If $H$ is an open subgroup of $G$, denote by $k\db{G}^H$ the set of $H$-fixed points of $k\db{G}$ under this action.  We have a linear map $\tn{Tr}_H^G: k\db{G}^H\to k\db{G}^G$ given by $\tn{Tr}_H^G(a):= \sum_{g\in G/H}{}^{g}a$, where ``$G/H$'' denotes a set of left coset representatives of $H$ in $G$.  If $B$ is a block of $k\db{G}$ with block idempotent $e$, a \emph{defect group} of $B$ is a closed subgroup $D$ of $G$, minimal with the property that 
$$e\in \bigcap_{N\trianglelefteq_{O}G}\tn{Tr}_{DN}^G(k\db{G}^{DN}).$$
The defect groups of a block are pro-$p$ subgroups of $G$, unique up to conjugation in $G$.

A block of $k\db{G}$ is precisely an indecomposable summand of $k\db{G}$, treated as a $k\db{G\times G}$-module with action $(g_1,g_2)\cdot x := g_1xg_2^{-1}$.  We may thus define the coinvariants of a block $B$: if $N$ is an open normal subgroup of $G$, we define $B_N := B_{N\times N}$.  For any given $N$, $B_N$ is a direct product of blocks of $k[G/N]$, though not necessarily a block.  As with $k\db{G}$, the natural projections $\varphi_{MN}, \varphi_N$ are algebra homomorphisms.  We recall some further results from \cite{JR}:

\begin{lem}[{\cite[Proposition 4.7 and Corollary 5.10]{JR}}]\label{lem:inverselimofblockswithcyclicdefectgrp}
	If the block $B$ of $k\db{G}$ has defect group $D$, then $B$ can be expressed as the inverse limit of a surjective inverse system of blocks $Z_N$ of $k[G/N]$ having defect group $DN/N$, as $N$ runs through some cofinal inverse system of open normal subgroups of $G$.  
\end{lem}

In \cite{JR} the blocks $Z_N$ of Lemma \ref{lem:inverselimofblockswithcyclicdefectgrp} were called ``$X_N$'', but in this article we will use the notation $X_N$ for something else.

\section{Blocks with cyclic defect groups}

From this point forward we suppose that $k$ is algebraically closed.  Recall that a profinite group is \emph{cyclic} if it possesses a dense cyclic abstract subgroup.  The cyclic pro-$p$ groups are the finite cyclic groups $\mathbb{Z}/p^n\mathbb{Z}$ ($n\in \mathbb{N}$) and the $p$-adic integers $\mathbb{Z}_p$.  We fix some notation.  Let $B$ be block of the profinite group $G$ having non-trivial cyclic defect group $D$.  Denote by $\mathcal{S}$ a set of representatives of the isomorphism classes of simple $B$-modules.  By \cite[Proposition 3.3]{FranZub}, there is a natural bijection between the indecomposable projective $B$-modules and the simple $B$-modules, given by sending the projective indecomposable $P$ to the simple module $P/\tn{Rad}(P)$.

Denote by $\mathcal{P} = \{P_S\,:\,S\in \mathcal{S}\}$ a set of representatives of the isomorphism classes of indecomposable projective $B$-modules, indexed in such a way that $P_S$ is the projective cover of $S$.  Write $B=\varprojlim_{N\in\mathcal{N}}Z_{N}$ as in Lemma \ref{lem:inverselimofblockswithcyclicdefectgrp}, with the cofinal set $\mathcal{N}$ chosen so that the block $Z_N$ has (cyclic) defect group $DN/N$.   For $N\in \mathcal{N}$, we denote by $\mathcal{S}_N, \mathcal{P}_N$ corresponding sets of representatives of simple modules and indecomposable projective modules for $Z_N$.

\begin{lem}\label{lem:Ifinite}
	The set $\mathcal{S}$ (hence also $\mathcal{P}$) is finite.
\end{lem}

\begin{proof}
	It follows from \cite[Theorem 11.1.1 and Theorem 11.1.3]{lin2} that $|\mathcal{S}_N|$ divides $p-1$.
	
	Assume for contradiction that $|\mathcal{S}|\geq p$ and fix a set $\{S_{1},\hdots,S_{p}\}$ of distinct simples in $\mathcal{S}$. By \cite[Corollary 4.9]{JR}, there is $N\in\mathcal{N}$ such that $\{S_{1},\hdots,S_{p}\}$ are distinct simple modules in $Z_{N}$. But this contradicts the first sentence.
\end{proof}

Using Lemma \ref{lem:Ifinite}, we restrict $\mathcal{N}$ further, assuming from now on that each $N\in \mathcal{N}$ acts trivially on every simple module in $B$.

\begin{lem}\label{lem:PN indec}
	For each $P = P_S\in \mathcal{P}$ and $N\in \mathcal{N}$, the module $P_N$ is non-zero and indecomposable.
\end{lem}

\begin{proof}
	Let $\pi:P\to S$ be the canonical projection.  The functor $(-)_N$ is right exact, and so we obtain a surjective map $\pi_N : P_N \to S_N = S$, and so $P_N\neq 0$.  By Lemma \ref{lem:charactofradicalofamodule}, $P/\tn{Rad}(P)$ surjects onto $P_N/\tn{Rad}(P_N)$.  But the former is simple by the first paragraph of this section, and so $P_N/\tn{Rad}(P_N)$ is simple.  It follows that $P_N$ is indecomposable.
\end{proof}

\begin{lem}\label{lem:B_Nisablock} 
	Let $B$ be a block of $G$ with cyclic defect group $D$. There is $N_{0}\trianglelefteq_{O}G$ acting trivially on each $S\in\mathcal{S}$ and such that  $B_{N}$ is a block for each $N\leq N_{0}$. 
\end{lem}

\begin{proof}
	By Lemma \ref{lem:Ifinite}, $\mathcal{S}$ is finite, and hence by \cite[Corollary 4.9]{JR} there is $N_0\in \mathcal{N}$ such that each simple module lies in $Z_{N_0}$.  But each block of $B_{N_0}$ contains at least one simple module, and hence $Z_{N_0} = B_{N_0}$.  The same holds for any $N\leqslant N_0$.
\end{proof}

\begin{lem}\label{lem:multiplicityofPi}
	Fix $P = P_S\in\mathcal{P}$. The multiplicity of $P$ as a factor of $B$ is finite.
\end{lem}

\begin{proof}
	By Lemma \ref{lem:PN indec}, $P_N$ is the projective cover of $S_N = S$ for each $N\in\mathcal{N}$.  The multiplicity of $P_N$ as a summand of $B_N$ is $\tn{dim}_{k}(S)$ by \cite[Theorem 7.3.9]{Web}, 
	
	and so $P$ has the same multiplicity as a factor of $B$.
\end{proof}

\begin{lem}\label{lem:Rad Pi over Soc Pi}
	For each $P\in \mathcal{P}$,
	$\tn{Rad}(P)/\tn{Soc}(P)\cong\varprojlim_{N}\tn{Rad}(P_{N})/\tn{Soc}(P_{N}).$
\end{lem}

\begin{proof}
	Note that by \cite[Corollary 6.3.4]{B} and our hypotheses on $N\in \mathcal{N}$, the indecomposable projective $B_N$-modules are not simple, so that in particular $\tn{Soc}(P_N)\leqslant \tn{Rad}(P_N)$.  The result now follows by applying Lemma \ref{lem:quotientpseudocompactmodule} with $U=\tn{Rad}(P),\ V=\tn{Soc}(P)$ and $W_{N}=\tn{Soc}(P_{N})$.
\end{proof}

\subsection{Indecomposable projective modules of blocks with cyclic defect group}

We maintain the notation fixed in the previous section.  We denote by $\mathcal{N}$ a cofinal set of open normal subgroups such that each $N\in \mathcal{N}$ acts trivially on every simple $B$-module, and such that $B_N$ is a block with defect group $DN/N$.  Observe that under these conditions, $\mathcal{S}$ can be canonically identified with $\mathcal{S}_{N}$, via the map sending $S\in \mathcal{S}$ to $S_N \iso S$.

Recall that a finite dimensional module is \emph{uniserial} if it has a unique composition series.  We call a pseudocompact $k\db{G}$-module $U$ \emph{pro-uniserial} if $U_N$ is uniserial for each $N\in \mathcal{N}$.  A simple $k\db{G}$-module is a \emph{composition factor} of $U$ if it is a composition factor of some $U_N$.

\begin{lem}\label{lem:tecnpropertyaoboutPi}
	Fix $P\in\mathcal{P}$. Then $\tn{Rad}(P)/\tn{Rad}^{2}(P) = T\oplus T'$, where $T,T'$ are non-isomorphic simple modules or zero.
\end{lem}

\begin{proof}
	The projective $B_N$-module $P_N$ is indecomposable by Lemma \ref{lem:PN indec}, so by \cite[Theorem 11.1.8]{lin1}, there are two unique uniserial submodules $X_{N},Y_{N}$ of $P_{N}$ such that $\tn{Rad}(P_{N})=X_{N}+Y_{N}$ and  $\tn{Soc}(P_{N})=X_{N}\cap Y_{N}$. 
	
	We will analyze two cases.  First suppose that at least one of the modules $X_{N},Y_{N}$ is simple for every $N$. Assume without loss of generality that $Y_{N}$ is simple for each $N$. Then $Y_N = \tn{Soc}(X_N)$ and in particular $Y_{N}\subseteq X_{N}$.  But then since $X_N$ is uniserial, $X_N/\tn{Rad}(X_N) = \tn{Rad}(P_N)/\tn{Rad}^2(P_N)$ is a simple module, which we call $T_N$.  The inverse system $\tn{Rad}(P_N)/\tn{Rad}^2(P_N)$ as $N$ varies is surjective by Lemma \ref{lem:charactofradicalofamodule}, and hence $T_N$ does not depend on $N$ and the inverse limit $\tn{Rad}(P)/\tn{Rad}^2(P) = T\iso T_N$.  This completes the first case.
	
	Now suppose that $X_N$ and $Y_N$ are not simple for every $N$ in some cofinal system inside $\mathcal{N}$, so that $\tn{Soc}(P_N) = X_N\cap Y_N = \tn{Rad}(X_N)\cap \tn{Rad}(Y_N)$.  The surjective map 
	$$\gamma_N : X_N + Y_N \to (X_N/\tn{Rad}(X_N))\oplus (Y_N/\tn{Rad}(Y_N))$$ 
	sending $x+y$ to $(x + \tn{Rad}(X_N)\,,\, y + \tn{Rad}(Y_N))$ is well-defined because $X_N\cap Y_N = \tn{Rad}(X_N)\cap \tn{Rad}(Y_N)$.  The kernel of $\gamma_N$ is $\tn{Rad}(X_N)+\tn{Rad}(Y_N)$.
	
	But by \cite[Theorem 11.1.8]{lin1}, the only maximal submodules of $X_N+Y_N$ are $X_N + \tn{Rad}(Y_N)$ and $\tn{Rad}(X_N)+Y_N$, so that 
	$$\tn{Rad}^2(P_N) = \tn{Rad}(X_N+Y_N) = \tn{Rad}(X_N)+\tn{Rad}(Y_N).$$
	
	Putting all this together and denoting by $T_N,T_N'$ the simple modules $X_N/\tn{Rad}(X_N)$, $Y_N/\tn{Rad}(Y_N)$ respectively, we have
	\begin{eqnarray*}
		\tn{Rad}(P_N)/\tn{Rad}^2(P_N) &=& (X_N + Y_N)/\tn{Rad}(X_N+Y_N)\\
		&=& (X_N+Y_N)/(\tn{Rad}(X_N)+\tn{Rad}(Y_N))\\
		&\iso& T_N\oplus T_N'.
	\end{eqnarray*}
	
	The simple modules $T_N, T_N'$ are non-isomorphic by \cite[Theorem 11.1.8]{lin2}.  As in the first case, taking limits we see that $\tn{Rad}(P)/\tn{Rad}^2(P) = T\oplus T'\iso T_N\oplus T_N'$.
\end{proof}

\begin{prop}\label{prop:existenceofpro-uniserialsubmodules}
	Fix $P\in\mathcal{P}$. There are unique pro-uniserial submodules $X$, $Y$ of $P$ satisfying the following properties:
	\begin{itemize}
		\item[1.] $X\cap Y=\tn{Soc}(P)$.
		\item[2.] $X+Y=\tn{Rad}(P)$.
		\item[3.]$\frac{\tn{Rad}(P)}{\tn{Soc}(P)}\cong\frac{X}{\tn{Soc}(P)}\oplus\frac{Y}{\tn{Soc}(P)}$, and the modules $\frac{X}{\tn{Soc}(P)}\mbox{, } \frac{Y}{\tn{Soc}(P)}$ have no composition factors in common.
	\end{itemize}
\end{prop}

\begin{proof}
	By Lemma \ref{lem:tecnpropertyaoboutPi}, $\tn{Rad}^{2}(P)$ is open in $P$, so we work within the cofinal system of those $N\in\mathcal{N}$ such that $I_{N}P\subseteq \tn{Rad}^{2}(P)$. By Lemma \ref{lem:tecnpropertyaoboutPi}, $\tn{Rad}(P)/\tn{Rad}^{2}(P)\cong T\oplus T'$, where $T,T'$ are non-isomorphic simple modules or zero. If one of them is 0, let it be $T'$.  In this case, set $X_{N}$ to be $\tn{Rad}(P_{N})$ and $Y_{N}$ to be $\tn{Soc}(P_{N})$. If $T,T'$ are both non-zero, then for each $N$ in the cofinal system, let $X_{N}$ be the maximal uniserial submodule of $\tn{Rad}(P_{N})$ such that $X_{N}/\tn{Rad}(X_{N})\cong T$ and let $Y_{N}$ be the maximal uniserial submodule of $\tn{Rad}(P_{N})$ such that $Y_{N}/\tn{Rad}(Y_{N})\cong T'$. By \cite[Theorem 11.1.8]{lin2}, $X_{N}$ and $Y_{N}$ are the unique submodules with this property.
	
	Since $\varphi_{MN}$ sends $\tn{Rad}(P_{N})$ onto $\tn{Rad}(P_{M})$ by Lemma \ref{lem:charactofradicalofamodule}, then $\varphi_{MN}(X_{N})=X_{M}$ and $\varphi_{MN}(Y_{N})\subseteq Y_{M}$ by the construction of $X_N, Y_N$ and using that $T,T'$ are non-isomorphic. So we may define $X:=\varprojlim_{N}\{X_{N},\varphi_{MN}\}$ and $Y:=\varprojlim_{N}\{Y_{N},\varphi_{MN}\}$.  We must check that $X$ and $Y$ satisfy Properties 1,2 and 3.  But by \cite[Theorem 11.1.8]{lin2}, the three properties are satisfied when we replace $X,Y,\tn{Rad}(P)$ and $\tn{Soc}(P)$ with $X_N,Y_N,\tn{Rad}(P_N)$ and $\tn{Soc}(P_N)$, respectively, and hence they also hold for  $X,Y,\tn{Rad}(P)$ and $\tn{Soc}(P)$ by taking limits. 
\end{proof}

We make two observations that will be required as we proceed:
\begin{itemize}
	\item Proposition \ref{prop:existenceofpro-uniserialsubmodules} implies in particular that $\tn{Rad}^{n}(P)$ is open in $P$ for each $n\in \mathbb{N}$ and $P\in \mathcal{P}$.
	\item From the proof of Proposition \ref{prop:existenceofpro-uniserialsubmodules} it follows that the maps $X_{N}\to X_{M}$ are surjective whenever $N\leq M$ and hence that the maps $\frac{X_{N}}{\tn{Soc}(P_{N})}\to\frac{X_{M}}{\tn{Soc}(P_{M})}$ are surjective. The maps $Y_{N}\to Y_{M}$ are surjective except perhaps when $Y_{N}=\tn{Soc}(P_{N})$. But in this case $\frac{Y_{N}}{\tn{Soc}(P_{N})}=\frac{Y_{M}}{\tn{Soc}(P_{M})}=0$, so the map $\frac{Y_{N}}{\tn{Soc}(P_{N})}\to\frac{Y_{M}}{\tn{Soc}(P_{M})}$ is surjective anyway. 
\end{itemize}

When $G$ is finite, the block $B$ has more structure. Namely, there are permutations $\rho,\sigma$ of $\mathcal{S}$ and the uniserial submodules $X,Y$ of each $P = P_S$ can be renamed $U_{S}$ and $V_{S}$ in a  such way that the distinct composition factors of $U_{S}$ are given by the $\rho$-orbit of $S$ as we descend the (unique) composition series of $U_{S}$, and similarly the distinct composition factors of $V_{S}$ are given by the $\sigma$-orbit of $S$. We lift this structure to profinite groups.



If ever $W$ is a pseudocompact $B$-module, denote by $\tn{Fac}(W)\subseteq\mathcal{S}$ the set of distinct representatives of the isomorphism classes of composition factors of $W$.  The following proposition is a pseudocompact version of \cite[Theorem 11.1.8]{lin2}:

\begin{prop}\label{prop:existenceofpro-uniserialsubmodulesUiVi}
	Denote by $X_S, Y_S$ the submodules of $P = P_S$ constructed in Proposition \ref{prop:existenceofpro-uniserialsubmodules}.
	There are two permutations $\rho,\sigma$ of $\mathcal{S}$  with the following property.  For each $S\in \mathcal{S}$, the submodules $X_S,Y_S$ can be renamed $U_{S},V_{S}$ in a such way that:
	\begin{itemize}
		\item the first $|\langle\rho\rangle\cdot S|$ composition factors of $U_{S}$ from the top are
		$$\frac{U_{S}}{\tn{Rad}(U_{S})}\cong \rho(S),\ \frac{\tn{Rad}(U_{S})}{\tn{Rad}^{2}(U_{S})}\cong \rho^{2}(S),\hdots ,\ \frac{\tn{Rad}^{|\langle\rho\rangle\cdot S|-1}(U_{S})}{\tn{Rad}^{|\langle\rho\rangle\cdot S|}(U_{S})}\cong \rho^{|\langle\rho\rangle\cdot S|}(S)= S,$$
		
		\item  the first $|\langle\sigma\rangle\cdot S|$ composition factors of $V_{S}$ from the top are
		$$\frac{V_{S}}{\tn{Rad}(V_{S})}\cong \sigma(S),\ \frac{\tn{Rad}(V_S)}{\tn{Rad}^{2}(V_{S})}\cong \sigma^{2}(S),\hdots ,\\ \frac{\tn{Rad}^{|\langle\sigma\rangle\cdot S|}-1(V_{S})}{\tn{Rad}^{|\langle\sigma\rangle\cdot S|}(V_{S})}\cong \sigma^{|\langle\sigma\rangle\cdot S|}(S)= S,$$
	\end{itemize}
	where $|\langle\rho\rangle\cdot S|,\ |\langle\sigma\rangle\cdot S|$ denote the sizes of the $\rho$-orbit $\langle\rho\rangle\cdot S$ of $S$ and the $\sigma$-orbit $\langle\sigma\rangle\cdot S$ of $S$. 
\end{prop}

\begin{proof}
	We maintain the notations from the proof of Proposition \ref{prop:existenceofpro-uniserialsubmodules}. Since $\mathcal{S}$ is finite, we may work within a cofinal inverse system of $N\in \mathcal{N}$ in which $\tn{Fac}(X_S)=\tn{Fac}((X_{S})_{N})$ and $\tn{Fac}(Y_S)=\tn{Fac}((Y_S)_{N})$ for each $S$.

	Fix some $N$ in our cofinal inverse system.  By the finite version of this result \cite[Theorem 11.1.8]{lin2}, there are permutations $\rho_{N},\sigma_{N}$ of $\mathcal{S}=\mathcal{S}_{N}$ satisfying the conclusions of the Lemma with respect to an appropriate renaming of $(X_S)_{N}$ and $(Y_S)_{N}$ as $(U_S)_{N}$ and $(V_S)_{N}$, for each $S$. 
	
	By how we chose $X,Y$ in Proposition \ref{prop:existenceofpro-uniserialsubmodules}, $\varphi_{N}(X_{S})=(X_S)_{N}$. Rename $X_{S}$ by $U_{S}$ if $(X_S)_{N}=(U_S)_{N}$ and by $V_{S}$ if $(X_S)_{N}=(V_S)_{N}$. Rename $Y_{S}$ to be the other one. This renaming is unambiguous unless $X_{S}=Y_{S}=\tn{Soc}(P_{S})$, in which case one may rename arbitrarily.
	
	Define $\rho:=\rho_{N}$ and $\sigma:=\sigma_{N}$. For each $S$, the condition that $\tn{Fac}(U_{S})=\tn{Fac}((U_S)_{N})$ implies that $I_{N}U_{S}\subseteq \tn{Rad}^{|\langle\rho\rangle\cdot S|}(U_{S})$, since otherwise not every isomorphism class of composition factor of $U_{S}$ would appear as a composition factor of $(U_S)_{N}$, and similarly with $V_{S}$. Hence, the first $|\langle\rho\rangle\cdot S|$ composition factors of $U_{S}$ are
	$$\frac{U_{S}}{\tn{Rad}(U_{S})}\cong \rho(S),\ \frac{\tn{Rad}(U_{S})}{\tn{Rad}^{2}(U_{S})}\cong \rho^{2}(S),\hdots ,\  \frac{\tn{Rad}^{|\langle\rho\rangle\cdot S|-1}(U_{S})}{\tn{Rad}^{|\langle\rho\rangle\cdot S|}(U_{S})}\cong \rho^{|\langle\rho\rangle\cdot S|}(S)=S,$$
	and similarly with $V_{S}$, as required.
\end{proof}

For each $N\in\mathcal{N}$, we define the permutations $\rho_{N}:=\rho$ and $\sigma_{N}:=\sigma$ of $\mathcal{S}_{N}=\mathcal{S}$.  By construction, $\rho_{N}$ and $\sigma_{N}$ satisfy \cite[Theorem 11.1.8]{lin2} with respect to the block $B_{N}$ of $G/N$.

\section{Brauer trees and Brauer tree algebras}

We maintain the notation from the previous sections, but reduce further our cofinal set of open normal subgroups: let $\mathcal{N}$ denote a set of open normal subgroups $N$ of $G$ such that $B_{N}$ is a block with cyclic defect group $DN/N$, $N$ acts trivially on each $S\in\mathcal{S}$, and $\tn{Fac}(U_{S})=\tn{Fac}((U_S)_{N})$, $\tn{Fac}(V_{S})=\tn{Fac}((V_S)_{N})$ for every $S\in\mathcal{S}$. If $D$ is a finite, then we demand further that $|D|=|DN/N|$, and if $D$ is $\mathbb{Z}_p$, then we demand further that $(|DN/N|-1)/|\mathcal{S}|>1$. 

We recall the notions of Brauer Trees and Brauer Tree Algebras.  Using the canonical identification $\mathcal{S}=\mathcal{S}_{N}$, $\rho_{N}:=\rho$ and $\sigma_{N}:=\sigma$ for each $N\in\mathcal{N}$, denote by $\Gamma(B_{N})$ the Brauer tree of $B_{N}$, meaning:

$\Gamma(B_{N})$ is a tree with vertices the $\rho$-orbits $\langle\rho\rangle\cdot S$ and the $\sigma$-orbits $\langle\sigma\rangle\cdot S$ of $\mathcal{S}$.  Its edges are the elements of $\mathcal{S}$, and the edge $S$ joins the $\rho$-orbit of $S$ and the $\sigma$-orbit of $S$. Unless $|DN/N|=p$ and $|\mathcal{S}|=p-1$, there is a unique exceptional vertex, to which we attribute the number $m_{N}=\frac{|DN/N|-1}{|\mathcal{S}|}$ (cf. \cite[Theorem 11.1.9]{lin2}).  There is a cyclic ordering $\gamma_{v}$ of the edges around the vertex $v=\langle\rho\rangle\cdot S$ given by:
$$\rho(S)\,,\, \rho^{2}(S)\,,\, \hdots,\, \rho^{|\langle\rho\rangle\cdot S|-1}(S)\,,\, \rho^{|\langle\rho\rangle\cdot S|}(S)=S.$$
There is a cyclic ordering $\gamma_{v}$ of the edges around the vertex $v=\langle\sigma\rangle\cdot S$ given by:
$$\sigma(S)\,,\, \sigma^{2}(S)\,,\, \hdots,\,  \sigma^{|\langle\sigma\rangle\cdot S|-1}(S)\,,\, \rho^{|\langle\sigma\rangle\cdot S|}(S)=S.$$

By \cite[Theorem 5.10.37]{Zim}, $B_{N}$ can be naturally identified with a Brauer tree algebra $A_N$ associated to $\Gamma(B_{N})$, which can be described as follows (we follow \cite[Definition 5.10.4]{Zim}):

\begin{enumerate} 
	\item There is a one-to-one correspondence between the edges of the tree and the simple $A_N$-modules,
	
	\item the top $P/\tn{Rad}(P)$ of each indecomposable projective $A_N$-module $P$ is isomorphic to its socle,
	
	\item the projective cover $P_{S}$ of the simple module corresponding to the edge $S$ is such that
	$$\tn{Rad}(P_{S})/\tn{Soc}(P_{S})\cong U^{v}_N(S)\oplus U^{w}_N(S)$$
	for two (possibly zero) uniserial $A_N$-modules $ U^{v}_N(S)$ and $U^{w}_N(S)$, where $v,w$ are the vertices adjacent to the edge $S$,
	
	\item if $v$ is not the exceptional vertex and if $v$ is adjacent to the edge $S$, then  $U^{v}_N(S)$ has $s(v)-1$ composition factors, where $s(v)$ is the number of edges adjacent to $v$,
	
	\item if $v$ is the exceptional vertex and if $v$ is adjacent to $S$, then $U^{v}_N(S)$ has $m_N\cdot s(v)-1$ composition factors.
	
	\item if $v$ is adjacent to $S$, then the composition factors of $U^{v}_N(S)$ are described as 
	$$\tn{Rad}^{j}(U^{v}_N(S))/\tn{Rad}^{j+1}(U^{v}_N(S))\cong  \gamma^{j+1}_{v}(S),$$
	\noindent for all $j$ smaller than the number of composition factors of $U^{v}_N(S)$.
\end{enumerate}

\begin{lem}\label{lem:BrauertreeGammaB_Ndescription}
	Via the canonical identification $\mathcal{S}=\mathcal{S}_{N}$, the Brauer trees $\Gamma(B_{N})$ are equal for each $N\in\mathcal{N}$, except for the multiplicity $m_{N}$.
\end{lem}

\begin{proof}
	Ignoring the exceptional vertex, the Brauer trees are clearly the same for each $N$, being completely determined by $\rho_{N} = \rho, \sigma_{N} = \sigma$ and $\mathcal{S}_{N}=\mathcal{S}$.  We need only check that the exceptional vertex (if it exists) is the same in each tree.  Our conditions on $\mathcal{N}$ imply that either no $\Gamma(B_{N})$ has an exceptional vertex, in which case there is nothing to check, or they all do.  Fix $N\leqslant M$ in $\mathcal{N}$ and let $v$ be the exceptional vertex of $\Gamma(B_M)$.   By the description above, the modules $U^{v}_{M}(S)$ are the only ones having strictly more than the size of the corresponding orbit composition factors. But since, via the natural identifications between $B_N$ and $A_N$, and between $B_M$ and $A_M$, $U^{v}_{N}(S)$ surjects onto $U^{v}_{M}(S)$, then $U^{v}_{N}(S)$ also has more than the size of the corresponding orbit composition factors.  Hence $v$ is also the exceptional vertex of $\Gamma(B_N)$, as required.
\end{proof}

\begin{defi}\label{def:brauertreeofB}
	Define the \emph{Brauer tree} of $B$ to be $\Gamma(B):=\Gamma(B_{N})$, for any $N\in\mathcal{N}$, except for the multiplicity $m$ of the exceptional vertex, which is defined to be $\frac{|D|-1}{|\mathcal{S}|}$ if $D$ is finite, or $\infty$ if $D$ is infinite.
\end{defi}

The next theorem is a pseudocompact version of \cite[Theorem 5.10.37]{Zim}:

\begin{teo}\label{teo:BraTreeAlg}
	Let $B$ be a block of a profinite group $G$ with cyclic defect group $D$. Then $B$ can be naturally identified with a Brauer tree algebra $A$ associated to $\Gamma(B)$, which can be described as follows:
	\begin{enumerate} 
		\item\label{teo:BrTreeAlgSIMPLESAREEDGES} There is a one-to-one correspondence between the edges of $\Gamma(B)$ and the simple $A$-modules.
		
		\item\label{teo:BrTreeAlgSOCLE} if the edge $S$ is adjacent to an exceptional vertex of infinite multiplicity, then the socle of the projective cover $P_S$ of the corresponding simple module is $0$.  Otherwise, $\tn{Soc}(P_S)\iso S$,
		
		\item\label{teo:BrTreeAlgRADOVERSOCTWOUNIS} 
		for each edge $S$ we have
		$$\tn{Rad}(P_S)/\tn{Soc}(P_{S})\cong U^{v}(S)\oplus U^{w}(S)$$
		for two (possibly zero) pro-uniserial modules $U^{v}(S)$ and $U^{w}(S)$, where $v,w$ are the vertices adjacent to the edge $S$,
		
		\item\label{teo:BrTreeAlgNUMFACSNONEXCEP} if $v$ is not the exceptional vertex and if $v$ is adjacent to the edge $S$ then $U^{v}(S)$ has $s(v)-1$ composition factors, where $s(v)$ is the number of edges adjacent to $v$,

		\item\label{teo:BrTreeAlgNUMFACSSEXCEP} if $v$ is the exceptional vertex with multiplicity $m$, and if $v$ is adjacent to $S$, then $U^{v}(S)$ has $m\cdot s(v)-1$ composition factors if $m$ is finite, or infinitely many  if $m=\infty$.
		
		\item\label{teo:BrTreeAlgCOMPFACS} if $v$ is adjacent to $S$ then the composition factors of $U^{v}(S)$ are described as 
		$$\tn{Rad}^{j}(U^{v}(S))/\tn{Rad}^{j+1}(U^{v}(S))\cong  \gamma_{v}^{j+1}(S),$$
		for all $j$ as long as $j$ is smaller than the number of composition factors of $U^{v}(S)$.
	\end{enumerate}
\end{teo}

\begin{proof}
	Property \ref{teo:BrTreeAlgSIMPLESAREEDGES} follows by the construction of $\Gamma(B)$. Property \ref{teo:BrTreeAlgRADOVERSOCTWOUNIS} follows from the construction of $\Gamma(B)$ and Proposition \ref{prop:existenceofpro-uniserialsubmodules}.
	
	Let $v$ be a non-exceptional vertex of $\Gamma(B)$. For each edge $S$ and each $N$, the module $U^v_N(S)$ has $s(v)-1$ composition factors and hence, since the maps $U^v_N(S)\to U^v_M(S)$ are surjective whenever $N\leqslant M$, then $U^v(S)$ also has $s(v)-1$ composition factors, confirming Property \ref{teo:BrTreeAlgNUMFACSNONEXCEP}.
	
	If $v$ is the exceptional vertex of $\Gamma(B)$ and the multiplicity $m$ is finite, then there is $N'\in\mathcal{N}$ such that $I_{N'}U^{v}(S)\subseteq \tn{Rad}^{n}(U^{v}(S))$, where $n=m\cdot s(v)$, because $\tn{Rad}^n(U^v(S)$ is open in $U^v(S)$. Since for each $N\leq N'$, the module $U_{N}^{v}(S)$ has $m\cdot s(v)-1$ composition factors, then $U^{v}(S)$ has $m\cdot s(v)-1$ composition factors.  If $v$ is the exceptional vertex and $m$ is infinite, then for each $n\in\NN$, there is $N\in\mathcal{N}$ such that $m_{N}\cdot s(v)-1 > n$. Since $U^{v}(S)$ surjects onto $U_{N}^{v}(S)$, it follows that $U^{v}(S)$ has at least $n$ composition factors for every $n$, and hence it has infinitely many.  These checks confirm Property \ref{teo:BrTreeAlgNUMFACSSEXCEP}.
	
	We may now check Property \ref{teo:BrTreeAlgSOCLE}.   Fix an edge $S$ and its projective cover $P = P_S$.  Properties \ref{teo:BrTreeAlgRADOVERSOCTWOUNIS}, \ref{teo:BrTreeAlgNUMFACSNONEXCEP} and \ref{teo:BrTreeAlgNUMFACSSEXCEP}, together with that fact that $\tn{Soc}(P) = \varprojlim \tn{Soc}({P}_N)$ has dimension at most $1$, imply that if $S$ is not connected to a vertex of infinite multiplicity, then $P$ is finite dimensional, and hence is isomorphic to ${P}_N$ for some $N\in \mathcal{N}$.  So by the finite result, $\tn{Soc}(P_S) \iso S$.  On the other hand, if $S$ is connected to a vertex of infinite multiplicity, then the map $\varphi_{MN}:P_N\to P_M$ restricted to $\tn{Soc}(P_N)$ is the zero map whenever $|DN/N|>|DM/M|$, and hence $\tn{Soc}(P) = \varprojlim \tn{Soc}(P_N) = 0$.
	
	It remains to check Property \ref{teo:BrTreeAlgCOMPFACS}, so fix $j$ smaller than the number of composition factors of $U^v(S)$ and some $n$ larger than $j$.  Since $\tn{Rad}^n(U^v(S))$ is open in $U^v(S)$, there is $N\in \mathcal{N}$ such that $I_NU^v(S)\subseteq \tn{Rad}^n(U^v(S))$.  For this $N$ we have
	\begin{align*}
		\tn{Rad}^{j}(U^{v}(S))/\tn{Rad}^{j+1}(U^{v}(S)) & \iso \tn{Rad}^{j}(U^{v}(S)_N)/\tn{Rad}^{j+1}(U^{v}(S)_N) \\
		& \iso 
		\tn{Rad}^{j}(U^{v}_N(S))/\tn{Rad}^{j+1}(U^{v}_N(S)) \\
		& \iso \gamma_v^{j+1}(S).
	\end{align*}
\end{proof}

\begin{teo}\label{teo:grpdeffin}
	Let $B$ be a block of a profinite group $G$ with cyclic defect group $D$. Then $D$ is a finite group if, and only if, $\tn{dim}_k(B)<\infty$.
\end{teo}

\begin{proof}
	By Lemma \ref{lem:multiplicityofPi}, $B$ is a finite direct sum of indecomposable projective modules, so that $B$ has finite dimension if, and only if, each indecomposable projective $B$-module is finite dimensional.  But by Theorem \ref{teo:BraTreeAlg} this is the case if, and only if, $D$ is finite.
\end{proof}

\section{Blocks with defect group $\mathbb{Z}_p$}

Theorem \ref{teo:grpdeffin} says that when the cyclic defect group $D$ of a block $B$ is finite, then $B$ is finite dimensional, and hence the blocks of a profinite group with finite cyclic defect group are precisely the blocks of finite groups with cyclic defect group.  It remains to study the case $D = \mathbb{Z}_p$.  We will show that there are very few possibilities for such blocks and that they are extremely well behaved.

A Brauer tree $\Gamma$ is of \emph{star type} if every edge of $\Gamma$ is connected to the exceptional vertex, as in the following diagram, wherein the exceptional vertex is black and the non-exceptional vertices are white, and the cyclic ordering around a vertex is given by reading the edges in counter-clockwise order:
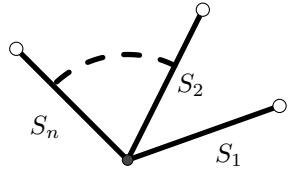
\begin{center}
	\definecolor{ffffff}{rgb}{1,1,1}
	\definecolor{uuuuuu}{rgb}{0.26666666666666666,0.26666666666666666,0.26666666666666666}
	\begin{tikzpicture}[line cap=round,line join=round,>=triangle 45,x=1cm,y=1cm]
		\draw [line width=2pt] (0,0)-- (-1.48,1.48) node[midway, below left] {$S_n$};
		\draw [line width=2pt] (0,0)-- (1,2) node[midway, right] {$S_2$};
		\draw [line width=2pt] (0,0)-- (2.02,0.72) node[midway, below right] {$S_1$};
		\draw [shift={(0,0)},line width=2pt,dash pattern=on 6pt off 8pt]  plot[domain=1.1071487177940904:2.356194490192345,variable=\t]({1*1.404250689869868*cos(\t r)+0*1.404250689869868*sin(\t r)},{0*1.404250689869868*cos(\t r)+1*1.404250689869868*sin(\t r)});
		\begin{scriptsize}
			\draw [fill=uuuuuu] (0,0) circle (2pt);
			\draw[color=uuuuuu] (-0.25,-0.345) ;
			\draw [fill=ffffff] (-1.48,1.48) circle (2.5pt);
			\draw [fill=ffffff] (1,2) circle (2.5pt);
			\draw [fill=ffffff] (2.02,0.72) circle (2.5pt);
		\end{scriptsize}
	\end{tikzpicture}
\end{center}

\begin{example}\label{ex: kZp}
	The simplest example of a block with defect group $\mathbb{Z}_p$ is $k\db{\mathbb{Z}_p}$ itself.  This algebra has only the trivial simple module $k$, and its Brauer tree is 
	\begin{center}
		\definecolor{ffffff}{rgb}{1,1,1}
		$$\begin{tikzpicture}[line cap=round,line join=round,>=triangle 45,x=1cm,y=1cm]
			\draw [line width=2pt] (-5,1) node[below left] {$\infty$} -- (-2,1) node[midway, below] {$k$};
			\begin{scriptsize}
				\draw [fill=black] (-5,1) circle (2.5pt);
				\draw [fill=ffffff] (-2,1) circle (2.5pt);
			\end{scriptsize}
		\end{tikzpicture}$$
	\end{center}
\end{example}

\begin{teo}\label{teo:blocktreestar}
	Let $B$ be a block of a profinite group $G$ with infinite cyclic defect group $D$. Then $\Gamma(B)$ is of star type with exceptional vertex of multiplicity $\infty$.
\end{teo}

\begin{proof}
    Let $P$ be a non-zero projective $k\db{G}$-module.  The module $P\res{D}$ is thus a non-zero projective $k\db{D}$-module (projectivity can be seen by applying \cite[Corollary 3.3]{Bru} to the algebras $k[DN/N]$ and projective modules $k[G/N]\res{DN/N}$).  But since $D$ is a pro-$p$ group, projective modules are free (as can be seen by taking limits), so that $P\res{D}$ is isomorphic to some product of copies of $k\db{D}$ and in particular $P$ has infinite dimension because $D$ is infinite.  If $\Gamma(B)$ is not of star-type, then it has an edge not connected to the exceptional vertex.  But now by definition of the Brauer tree algebra, the projective cover of the corresponding simple $B$-module is finite dimensional, contradicting the above observation.
\end{proof}

Theorem \ref{teo:blocktreestar} tells us that there are very few Morita equivalence classes of blocks with defect group $\mathbb{Z}_p$.  We finish by observing briefly that they are extremely well-behaved as algebras: they are examples of what one might call a pseudocompact Nakayama algebra (cf.\ \cite{Nakayama}), they have global dimension 1 (cf.\ \cite[\S 3]{Bru}), and they are Morita equivalent (cf.\  \cite{GabrielIndecs2}) to completed path algebras (\cite[Definition 2.2]{DerksenWeymanZelevinsky}) of finite quivers with a very simple form: 

\begin{cor}\label{cor:props of blocks with def Zp}
	Let $B$ be a block of a profinite group, having defect group $\mathbb{Z}_p$.  Then
	\begin{enumerate}
		\item The indecomposable projective $B$-modules are pro-uniserial.
		\item $B$ has global dimension $1$.
		\item If $B$ has $n$ simple modules, then $B$ is Morita equivalent to the completed path algebra $k\db{Q}$, where $Q$ is an oriented cycle of length $n$.
	\end{enumerate}
\end{cor}

\begin{proof}
	Part 1 is immediate from Theorem \ref{teo:blocktreestar}.  Suppose the Brauer tree of $B$ has the form given in the diagram before Example \ref{ex: kZp}.  Using the symmetry of the Brauer tree, in order to check that $B$ has global dimension $1$, it is sufficient to check that the kernel $K$ of the projective cover $P_{S_1}\to S_1$ is projective.  But $K$ is pro-uniserial and its composition factors, starting at the top, are $S_2, S_3,\hdots, S_n, S_1, S_2, \hdots$, so that $K$ is isomorphic to $P_{S_2}$.  This argument also shows that the quiver of the corresponding completed path algebra is an oriented cycle.
\end{proof}

\bibliographystyle{abbrv}
\bibliography{refpap}

\end{document}